\documentclass{amsart}

\newtheorem{theorem}{Theorem}[section]
\newtheorem{lemma}[theorem]{Lemma}

\theoremstyle{definition}
\newtheorem{definition}[theorem]{Definition}

\theoremstyle{remark}
\newtheorem{remark}[theorem]{Remark}

\numberwithin{equation}{section}
\numberwithin{figure}{section}

\usepackage{graphicx}
\usepackage{epsfig}
\usepackage{epsf}

\newtheorem{corollary}{Corollary}[section]

\newcommand{\cN}{{\cal N}}

\newcommand{\eps}{\epsilon}
\newcommand{\bitem}{\begin{itemize}}
\newcommand{\eitem}{\end{itemize}}
\newcommand{\goto}{\rightarrow}

\newcommand{\sgn}{\mbox{sgn}}
\newcommand{\beqn}{\begin{equation}}
\newcommand{\eeqn}{\end{equation}}
\newcommand{\balign}{\begin{align}}
\newcommand{\ealign}{\end{align}}

\def\cal{\mathcal}

\def\RR{{\mathbb R}}

\def\po{\RR_+^N}
\def\h{H^N}
\def\na{{\cal N}(A)}

\def\psiwpo{\psi_W^{\RR_+}}
\def\psiwh{\psi_W^{H}}
\def\psispo{\psi_S^{\RR_+}}

\begin{document}

\title[Faces of the Randomly-Projected Hypercube and Orthant]{Counting the Faces of Randomly-Projected Hypercubes and Orthants, with Applications}

\author{David L. Donoho}
\address{Department of Statistics, Stanford University}
\curraddr{Department of Statistics, Stanford University}
\email{donoho@stanford.edu}
\thanks{The authors would like to thank the Isaac Newton Mathematical
Institute at Cambridge University for hosting
the programme
"Statistical Challenges of High Dimensional Data" in 2008.
and Professor D.M.
Titterington for organizing this programme.
DLD acknowledges support 
from NSF DMS 05-05303 and a Rothschild Visiting Professorship
at the University of Cambridge.}

\author{Jared Tanner}
\address{School of Mathematics, University of Edinburgh}
\curraddr{School of Mathematics, University of Edinburgh}
\email{jared.tanner@ed.ac.uk}
\thanks{JT acknowledges support from the Alfred P. Sloan Foundation
  and thanks John E. and Marva M. Warnock for their generous support
  in the form of an endowed chair.}



\date{May 2008}


\begin{abstract}
Let $A$ be an $n$ by $N$ real valued random matrix, and $\h$ denote the
$N$-dimensional hypercube.  For numerous random matrix ensembles,
the expected number of $k$-dimensional 
faces of the random $n$-dimensional zonotope $A\h$
obeys the formula  $E f_k(A\h) /f_k(\h) = 1-P_{N-n,N-k}$, where
$P_{N-n,N-k}$ is a fair-coin-tossing probability:
\[
P_{N-n,N-k} \equiv  \mbox{Prob}
\{ N-k-1 \mbox{ or fewer successes in }  N-n-1 \mbox{ tosses  }\}.
 \]
The formula applies, for example, where the columns of $A$ 
are drawn i.i.d. from an absolutely continuous  symmetric
distribution.  The formula exploits Wendel's 
Theorem\cite{We62}.  


Let $\po$ denote the positive orthant; the
expected number of $k$-faces of  the random cone$A \po$ obeys
$ {\cal E} f_k(A\po) /f_k(\po) = 1 - P_{N-n,N-k}$.
The formula applies to numerous matrix ensembles, including those
with iid random columns from an absolutely continuous, centrally symmetric
distribution.  

The probabilities
$P_{N-n,N-k}$ change rapidly  from nearly 0 to nearly 1
near $k \approx 2n-N$. 
Consequently, there is an asymptotically sharp threshold in
the behavior of face counts of the projected hypercube; thresholds known for
projecting the simplex and the cross-polytope, occur at very different locations.
We briefly consider face counts  of the projected orthant when
$A$ does not have mean zero; these do behave similarly to
those for the projected simplex.  We consider non-random projectors
of the orthant; the 'best possible' $A$ is the one associated with
the first $n$ rows of the Fourier matrix.


These geometric face-counting results 
have implications  for signal processing, information theory, inverse problems, and 
optimization.  Most of these flow in some way from the fact that face counting
is related to conditions for uniqueness of solutions of 
underdetermined systems of linear equations.

\noindent  
a)  A vector in $\po$ is called $k$-sparse
if it has at most $k$ nonzeros. For such a $k$-sparse vector $x_0$, let $b = Ax_0$,
where $A$ is a random matrix ensemble covered by our results.
With probability $1 - P_{N-n,N-k}$
the inequality-constrained system $Ax=b$, $x\ge 0$ 
has $x_0$ as its unique nonnegative solution.
This is so, even if $n < N$, so that the system $Ax=b$ is underdetermined.

\noindent
b) A vector in the hypercube $\h$
will be called $k$-simple if all entries except at most $k$ 
are at the bounds 
0 or 1.  For such a $k$-simple vector $x_0$, let $b = Ax_0$,
where $A$ is a random matrix ensemble covered by our results.
With probability  $1 - P_{N-n,N-k}$
the inequality-constrained system $Ax=b$, $x \in \h$ 
has $x_0$ as its unique solution in the hypercube.

\end{abstract}

\maketitle



\section{Introduction}\label{sec:intro}
\setcounter{equation}{0}
\setcounter{table}{0}
\setcounter{figure}{0}

There are 3 fundamental {\em regular} polytopes in $\RR^N$, $N\ge 5$: 
the hypercube $\h$,
the cross-polytope $C^N$, and the simplex $T^{N-1}$.
For each of these, projecting the vertices into 
$\RR^n$, $n <N$,  yields the vertices of a new polytope; in fact, every 
polytope in $R^n$ can be generated by rotating the  
simplex $T^{N-1}$ and orthogonally projecting on the 
first $n$ coordinates, for some
choice of $N$ and of $N$-dimensional rotation. Similarly, 
every centro-symmetric polytope can be generated by projecting the 
cross-polytope, and every zonotope by projecting the hypercube.

\subsection{Random polytopes}\label{subsec:polytope}
Choosing the projection $A$
at random has become popular. 
Let $A$ be an $n\times N$ uniformly distributed random
orthogonal projection, obtained by first applying a uniformly-distributed
rotation to $\RR^N$ and then projecting on the
first $n$ coordinates. Let $Q$ be a polytope in $\RR^N$.  Then 
$AQ$ is a random polytope in $\RR^n$.   Taking $Q$ in turn from each of the
three families of regular polytopes we get three arenas for scholarly study:
\bitem
\item 
Random polytopes of the form $AT^{N-1}$  were
first studied by  Affentranger and Schneider \cite{AffenSchnei} and by  Vershik and Sporyshev \cite{VerSpor};
\item  Random polytopes of the form
$AC^N$  were first studied extensively by 
Borozcky and Henk \cite{BoroHenk};
\item The random zonotope $A\h$ will 
be heavily studied in this paper; beginnings of a literature on 
zonotopes can be found in \cite{Bolker,Bary}.
\eitem
Such random polytopes can have face lattices
undergoing abrupt changes in 
properties as dimensions change only slightly.
In the case of $AT^{N-1}$ and $AC^N$, previous work by the authors 
\cite{Do05_polytope,DoTa05_polytope,DoTa07_CF,DoTa07_FiniteN} 
documented the following {\t threshold} phenomenon. 
(Our work built on fundamental formulas developed by Affentranger and 
Schneider \cite{AffenSchnei} and used an asymptotic
framework pioneered by  Vershik and Sporyshev \cite{VerSpor},
who pointed to the first such threshold effect).
Let $f_k(Q)$ denote the number of $k$-dimensional
faces of polyhedron $Q$.  It turns out that for large $n$,
the number of $k$-dimensional faces of $f_k(AQ)$
might either be approximately equal to $f_k(Q)$ or else significantly smaller,
depending on the size of $k$ relative to a threshold depending on
the ratio of  $n$ to $N$.

To make this precise, consider the following {\it proportional-dimensional} asymptotic framework.
A {\it dimension specifier} is a triple of integers $(k,n,N)$,  representing
a `face' dimension $k$, a `small' dimension $n$ and  a `large' dimension $N$;  $k < n < N$.
 For fixed $ \delta, \rho \in (0,1)$,
consider sequences of dimension specifiers, indexed by $n$, and obeying
\begin{equation} \label{eq:propdim}
k_n/n \goto\rho\quad \mbox{ and }\quad n/N_n \goto\delta .
\end{equation}
For such sequences the small dimension $n$ is held proportional to the large dimension $N$
as both dimensions grow. We omit subscripts on $k_n$ and $N_n$ when possible.
For $Q = T^{N-1}$, $C^N$,  the papers
\cite{Do05_polytope,DoTa05_polytope,DoTa07_CF,DoTa07_FiniteN} 
exhibited thresholds $\rho(\delta;Q)$  for 
for the ratio between the expected number of faces of the low-dimensional  polytope $AQ$ and 
the number of faces of the high-dimensional polytope $Q$:
\begin{equation}\label{eq:weak}
\lim_{n \goto \infty} \frac{{\cal E}f_k(AQ)}{f_k(Q)}  \quad 
\left\{\begin{array}{cl}
=1 & \rho<\rho_W(\delta;Q) \\
<1 & \rho>\rho_W(\delta;Q)
\end{array}\right. .
\end{equation}
(In this relation, we take a limit as $n \goto \infty$ along some
sequence obeying the proportional-dimensional constraint (\ref{eq:propdim})).
In words, the random object $AQ$ has roughly as 
many $k$-faces as its generator $Q$, for $k$ below a threshold; and has 
noticeably fewer $k$-faces than $Q$, for $k$ 
above the threshold.  The threshold functions are defined in terms
of Gaussian integrals and other special functions, and can be calculated 
numerically.

These phenomena, described here from the viewpoint of
combinatorial geometry, have surprising consequences in probability theory,
information theory and 
signal processing; see \cite{Do05_signal,DoTa05_signal,DoTa07_CF}, and 
Section \ref{sec:cs} below.

\subsection{Random Zonotopes}
Missing from the above picture is information about the third family of
regular polytopes, the hypercube.  B\"{o}r\"{o}czky and Henk 
\cite{BoroHenk} discussed it in passing, but only considered 
the asymptotic framework 
where the small dimension $n$ is held fixed  while the large 
dimension $N\goto\infty$.
In that framework, the threshold phenomenon is not visible.
In this paper, we again consider the 
proportional-dimensional case   (\ref{eq:propdim}) and prove the 
following.

\begin{theorem}[`Weak' Threshold for Hypercube]\label{thm:main_weak_hypercube}
Let
\begin{equation}\label{eq:rhoW}
\rho_W(\delta;\h):=\max{(0,2-\delta^{-1})}.
\end{equation}
For $\rho$,$\delta$ in $(0,1)$, consider a sequence  of dimension specifiers  $(k,n,N)$ obeying 
(\ref{eq:propdim}). 
Let $A$ denote a uniformly-distributed random orthogonal projection from $\RR^{N}$ to $\RR^n$.
\begin{equation} \label{eq:hyp_weak}
\lim_{n \goto \infty}  \frac{{\cal E}f_k(A\h)}{f_k(\h)} = 
\left\{\begin{array}{cl}
1, & \rho<\rho_W(\delta,\h) \\
0, & \rho>\rho_W(\delta,\h)
\end{array}\right. .
\end{equation}
\end{theorem}

Thus we prove a sharp discontinuity in the behavior of the face 
lattices of random 
zonotopes; the location  of the threshold is precisely identified.  
(Such sharpness of the 
phase transition is also observed empirically for (\ref{eq:weak}) above;
to our knowledge, a proof of discontinuity has not yet been published in
that setting. ) Our use of the modifier `weak' and the subscript $W$ on $\rho$
matches usage in the previous cases  $T^{N-1}$ and $C^N$.

Although this result
has been stated in the language of combinatorial convexity, 
as with the earlier results for $AT^{N-1}$ and $AC^N$,  there are
implications for applied fields including optimization and signal 
processing, see Section \ref{sec:cs} below.

\subsection{More General Notion of Random Projection}
In fact, Theorem \ref{thm:main_weak_hypercube} 
is only the tip of the iceberg.  The ensemble of random matrices
used in that result - uniformly distributed random orthoprojector - 
is only one example of a random matrix ensemble for  which the conclusion
(\ref{eq:hyp_weak}) holds.  
As it turns out, what really matters
are the statistical properties of the nullspace of $A$.

\begin{definition}[Orthant-Symmetry]\label{def:os}
Let $B$ be a random $N-n$ by $N$ matrix such that
for each diagonal matrix $S$ with diagonal in $\{-1,1\}^N$, 
and for every measurable set $\Omega$, 
\[
\mbox{Prob}\{BS\in\Omega\}=\mbox{Prob}\{B\in\Omega\}.
\]
Then we say that $B$ is an {\em orthant-symmetric} random matrix.  Let 
$V_B$ be the linear span of the rows of $B$.  If $B$ is an 
orthant-symmetric random matrix we say that $V$ is an {\it orthant-symmetric 
random subspace}.  
\end{definition}

\begin{remark}[Orthant-Symmetric Ensembles]
The following ensembles of random matrices are orthant-symmetric:
\begin{itemize}
\item {\it Uniformly-distributed Random orthoprojectors} from $\RR^N$ to
$\RR^{N-n}$; implicitly this was the example considered earlier.
\item {\it Gaussian Ensembles}. A random matrix $B$ with entries chosen
from a Gaussian zero-mean distribution,
i.e. such that the $(N-n) \cdot N$-element vector
$vec(B)$ is $N(0,\Sigma)$ with $\Sigma$ a nondegenerate covariance
matrix.
\item {\it Symmetric i.i.d. Ensembles.} 
Matrices with entries sampled i.i.d. from a symmetric probability distribution; examples include
Gaussian $N(0,1)$, uniform on $[-1,1]$, 
uniform from the set $\{-1,1\}$, and from the set $\{-1,0,1\}$ where $-1$ and $1$ 
have equal non-zero probability.
\item {\it Sign Ensembles.} For {\em any} fixed generator matrix $B_0$, let 
the random matrix
$B=B_0S$ where $S$ is a random diagonal matrix with entries drawn 
uniformly from $\{-1,1\}$.
\end{itemize}
\end{remark}

New orthant-symmetric ensembles can  be created from an
existing one by multiplying on the left by an arbitrary random matrix
$T$ which is stochastically independent of $B$, and multiplying
on the right by a random diagonal matrix $R$ also stochastically independent
of $B$ and $T$: thus $B' = TBR$ inherits orthant symmetry from $B$.

\begin{definition}[General Position]\label{def:gp}
Let $B$ be a random $N-n$ by $N$ matrix such that every subset of
$N-n$ columns is almost surely linearly independent.  Let 
$V_B$ be the linear span of the rows of $B$. 
We say that $V_B$ is a {\it generic} random subspace.  
\end{definition}

Many orthant-symmetric ensembles from our list create
generic row spaces: 
\bitem
\item Uniformly-distributed random orthoprojectors;
\item Gaussian Ensembles; 
\item Symmetric iid ensembles having an absolutely continuous distribution;
\item Sign Ensembles with generator matrix $B_0$ having its columns in general position;
\eitem
Define a {\it censored symmetric iid ensemble} 
as a symmetric iid ensemble from which we discard realizations $B$
where the columns happen to be not in general position. Censoring
a symmetric iid ensemble made from the Bernoulli $\{-1,+1\}$
coin tossing distribution produces a new random matrix model $\tilde{B}$
whose realizations are in general position with probability one.
(The probability of a censoring event is
exponentially small in $N$, \cite{RV08}).

\begin{theorem}[`Weak' Threshold for Hypercube ]\label{thm:second_weak_hypercube}
Let the random matrix $A$ have a random nullspace which
is orthant symmetric and generic.  In the proportional-dimensional
framework  (\ref{eq:propdim}) the
random zonotope $A\h$ obeys the same conclusion (\ref{eq:hyp_weak}) as in
Theorem \ref{thm:main_weak_hypercube}.
\end{theorem}

In a sense, this theorem extends the conclusion of Theorem \ref{thm:main_weak_hypercube}
to vastly more cases .  It has been
previously observed that some results known for the Goodman-Pollack
random orthoprojector model actually extend to other ensembles
of random matrices. 
It was observed for the simplex by 
Affentranger and Schneider \cite{AffenSchnei}, and proven by 
Baryshnikov and Vitale \cite{BaryVitale, Bary}, that face-counting results known for
uniformly-distributed random orthoprojectors follow as well 
 for Gaussian iid matrices $A$. 

Our extension of Theorem 
\ref{thm:main_weak_hypercube} from orthoprojectors to 
orthant-symmetric null spaces in Theorem \ref{thm:second_weak_hypercube}
follows this program.   However, it is a vastly larger extension.

\subsection{Random Cone}\label{subsec:cone}

Convex \underline{cones} provide another type of fundamental 
polyhedral set.  Amongst these, the simplest and most natural is 
the positive orthant $P=\po$.  The image of a cone under projection 
$A$: $\RR^N\goto\RR^n$ is again a cone $K=AP$. 
Typically the cone  has $f_0(K)=1$ vertex (at 0), and $f_1(K)=N$ extreme rays, etc.  
In fact, every such pointed cone in  $\RR^n$ can be generated
as a projection of the positive orthant, with an appropriate orthogonal
projection from an appropriate $\RR^N$.

As with the polytopes models,  surprising threshold
phenomena can arise when the projector is random.

\begin{theorem}[`Weak' Threshold for Orthant]\label{thm:main_weak_orthant}
Let $A$ be a random matrix  
whose nullspace is an orthant-symmetric and generic random subspace. 
In the proportional-dimensional
framework  (\ref{eq:propdim}) we have
\begin{equation} \label{eq:weak_main}
\lim_{n \goto \infty}   \frac{{\cal E}f_k(A\po)}{f_k(\po)} = 
\left\{\begin{array}{cl}
1, & \rho<\rho_W(\delta;\po) \\
0, & \rho>\rho_W(\delta;\po)
\end{array}\right. 
\end{equation}
with $\rho_W(\delta;\po)\equiv\rho_W(\delta;\h)$ as defined in 
(\ref{eq:rhoW}).
\end{theorem}

Here the threshold for the orthant is at precisely the same place
as it was for the hypercube.
Theorem \ref{thm:main_weak_orthant} is proven in Section \ref{sec:asymp}, 
and there are significant implications in optimization and signal 
processing briefly discussed in Section~\ref{sec:cs}.

\subsection{Exact equality in the number of faces}

Our focus in Sections \ref{subsec:polytope}-\ref{subsec:cone} has been 
on the `weak' agreement of ${\cal E}f_k(AQ)$ with $f_k(Q)$; we have seen in
the proportional-dimensional framework, for $\rho$ below
threshold $\rho_W(\delta;Q)$,  we have limiting relative equality:
\[
  \frac{{\cal E}f_k(A\po)}{f_k(\po)}  \goto 1,  \qquad n \goto \infty.
\]
We now 
focus on the `strong' agreement; it turns out that 
in the proportional dimensional framework, for $\rho$ below
a somewhat lower threshold $\rho_S(\delta;Q)$, we actually have
exact equality with overwhelming probability:
\begin{equation} \label{eq:StrongLim}
\mbox{Prob}\{f_k(Q)=f_k(AQ)\} \goto 1,  \qquad n \goto \infty.
\end{equation}
The existence of such `strong' thresholds  for $Q=T^{N-1}$ and $Q=C^N$ was proven in
\cite{Do05_polytope, DoTa05_polytope}, which exhibited 
thresholds $\rho_S(\delta; Q)$ below which (\ref{eq:StrongLim}) occurs.
These ``strong thresholds'' and the previously
mentioned ``weak thresholds'' (\ref{eq:weak}) are depicted in 
Figure \ref{fig:simplex_crosspolytope}.
A similar strong threshold also holds for 
the projected orthant.

\begin{theorem}[`Strong' Threshold for Orthant]\label{thm:main_strong}
Let 
\begin{equation} \label{eq:Shannon}
H(\gamma):=\gamma\log(1/\gamma)-(1-\gamma)\log(1-\gamma)
\end{equation}
denote the 
usual (base-$e$) Shannon Entropy.  Let \begin{equation}\label{eq:psis}
\psispo(\delta,\rho):=H(\delta)+\delta H(\rho)-(1-\rho\delta)\log 2 .
\end{equation}
For $\delta\ge 1/2$, let 
$\rho_S(\delta;\po)$ denote the zero crossing of $\psispo(\delta,\rho)$.
In the proportional-dimensional framework (\ref{eq:propdim}) 
with $\rho<\rho_S(\delta;\po)$
\begin{equation} \label{eq:strongLim}
\mbox{Prob} \{ f_k(A\po)=f_k(\po)\} \goto 1,\quad\quad\mbox{as}\;\;
n\goto\infty.
\end{equation}
\end{theorem}

The threshold $\rho_W(\delta;Q)$ for $Q=\po$ and $\h$, 
and $\rho_S(\delta;\po)$ are depicted in Figure \ref{fig:phase}.

\begin{figure}[h!]
\begin{tabular}{c}
\psfig{figure=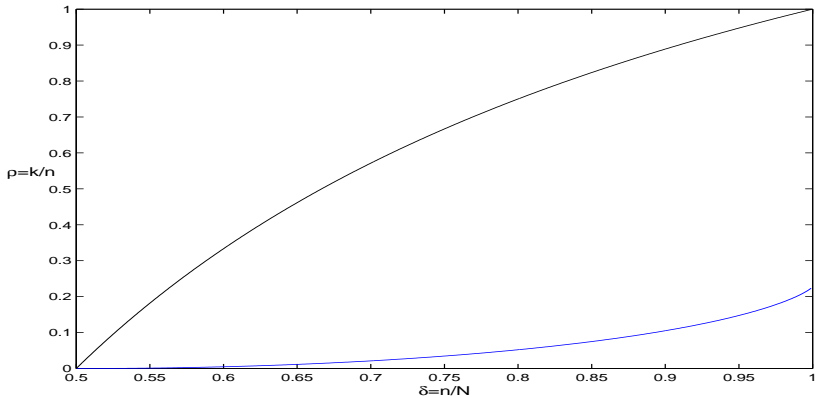,height=2.1in,width=4.3in} 
\end{tabular}
\caption{The `weak' thresholds, $\rho_W(\delta;\h)$ and $\rho_W(\delta;\po)$ 
(black), and a lower bound on the strong threshold for the positive orthant, 
$\rho_S(\delta;\po)$ (blue).}\label{fig:phase}
\end{figure}

In contrast to the projected simplex, cross-polytope, and orthant,
for the hypercube,  there is 
no nontrivial regime where a phenomenon like (\ref{eq:StrongLim}) can occur.  

\begin{lemma}[Zonotope Vertices]\label{thm:strong_zonotope}
Let $A$ be an $n\times N$ matrix, and let $\h$ be the $N$ dimensional 
hypercube.  \[
f_k(A\h)<f_k(\h), \qquad k=0,1,2,\ldots n .
\]
\end{lemma}
\begin{proof}[Proof of Theorem \ref{thm:strong_zonotope}]
In fact, we will show that $A\h$ always has 
fewer than $2^N$ vertices. This immediately implies the full result.
There exists a $w\in\na$ with $w\ne 0$.  
$\h$ has a vertex $x_0$ obeying
\[
x_0(i):=\left\{\begin{array}{ll} 0 & \sgn(w(i))>0, \\ 1 & \mbox{else}.
\end{array}\right.
\]
Let $x_{t}:=x+t w$ with $t>0$.  For $t$ 
sufficiently small $x_t$ is in the interior of $\h$, and 
by construction $Ax_0=Ax_{t}$.  Invoking  
Lemma \ref{lem:face_transverse_h}, $x_0$ is not a vertex 
of $A\h$, and $f_0(A\h)<f_0(\h)$.  
\end{proof}

Although this proof only highlights 
a single vertex of $\h$ that is interior to $A\h$, it is clear from its 
construction that there are typically many such lost vertices.
Theorem \ref{thm:main_strong} is proven in Section \ref{sec:proof_strong}.

\subsection{Exact Non-Asymptotic Results}

We have so far exclusively used the Vershik-Sporyshev
proportional-dimensional asymptotic framework; this makes for
the most natural comparisons between results for the three
families of regular polytopes.  However, for the positive
orthant and hypercube, something truly remarkable happens: there is
a simple exact expression for finite $N$ which connects to a beautiful 
result in geometric probability.

\begin{theorem}[Wendel, \cite{We62}]\label{thm:wendel}
Let $M$ points in $\RR^m$ be drawn i.i.d. from a centro-symmetric 
distribution such that the points are in general position, 
then the probability that all the points fall in some half space is
\begin{equation}\label{eq:wendel}
P_{m,M}=2^{-M+1}\sum_{\ell=0}^{m-1}
{M-1 \choose \ell}.
\end{equation}
\end{theorem}

This elegant result is often presented as simply a piece of recreational
mathematics.  In our setting, it turns out to be truly powerful, because
of the following identity.

\begin{theorem} \label{thm:isom}
Let $A$ be an $n \times N$ random matrix 
with an orthant-symmetric and generic random nullspace.
\begin{equation} \label{eq:exact}
 \frac{{\cal E}f_k(A\po)}{f_k(\po)} =  1 - P_{N-k,N-n} .
\end{equation}
\end{theorem}

Symmetry implies a similar identity for the hypercube:
 
\begin{theorem} \label{thm:finite_samp_hyp}
Let $A$ be a random matrix with
an orthant-symmetric and generic random nullspace.
\begin{equation}
 \frac{{\cal E}f_k(A\h)}{f_k(\h)} =  1 - P_{N-n,N-k}.
 \end{equation}
\end{theorem}

These formulae are not at all asymptotic or approximate.
But all the earlier asymptotic results derive from them.
Theorem \ref{thm:isom} is proven in Section \ref{sec:isom}, and 
the symmetry argument for Theorem \ref{thm:finite_samp_hyp}
is formalized in Lemma \ref{lem:hyper_orth} and proven in 
Section \ref{sec:orth_proof}.

\subsection{Contents}
Theorem \ref{thm:main_weak_orthant} is proven in Section 
\ref{sec:asymp},
Theorems \ref{thm:main_weak_hypercube} and \ref{thm:second_weak_hypercube} 
are proven in Section \ref{sec:weak_hypercube_proof}, and 
Theorem \ref{thm:main_strong} is proven in Section \ref{sec:proof_strong};
each using the classical Wendel's Theorem \cite{We62}, 
Theorem \ref{thm:wendel}.  Their relationships with existing results 
in convex geometry and matroid theory are discussed in Section 
\ref{sec:existing}, the the implications of these results for information 
theory, signal processing, and optimization are briefly discussed in 
Section \ref{sec:cs}.

\section{Proof of main results}

Our plan is to start with the key non-asymptotic exact identity (\ref{eq:exact})
and then derive 
from it Theorem \ref{thm:main_weak_orthant}
by asymptotic analysis of the probabilities in Wendel's Theorem.
We then infer Theorem \ref{thm:second_weak_hypercube} and
later Theorem \ref{thm:main_strong} follows in Section 
\ref{sec:proof_strong}.

\subsection{Proof of Theorem \ref{thm:isom}}\label{sec:isom}

Here and below we follow the convention that, if we don't give
the proof of a lemma or corollary immediately following its statement, 
then the proof can be found in Section \ref{sec:MoreProofs}.

Our proof of the key formula (\ref{eq:exact}) starts with the 
following observation on the expected number of $k$-faces 
of $\po$.  
\begin{equation} \label{eq:aveFace}
 \frac{{\cal E}f_k(A\po)}{f_k(\po)} =  
\mbox{Ave}_F\left[ \mbox{Prob}\{ AF \mbox{ is a $k$-face of } A\po \} \right].
 \end{equation}
Here $Ave_F$ denotes "the arithmetic mean over all $k$-faces of $\po$.

Because of (\ref{eq:aveFace}) we will be implicitly averaging
across faces below.
As a calculation device we 
suppose that all faces are statistically equivalent;
this allows us to study one $k$-face, 
and yet compute the average across all $k$-faces.

\begin{definition}[Exchangable columns]\label{def:exchangable}
Let $A$ be a random $n$ by $N$ matrix such that
for each permutation matrix $\Pi$, 
and for every measurable set $\Omega$, 
\[
\mbox{Prob}\{A\in\Omega\}=\mbox{Prob}\{A\Pi\in\Omega \}
\]
Then we say that $A$ has exchangeable columns.
\end{definition}

Below we assume without loss of generality that $A$ has exchangeable columns.
Then   (\ref{eq:aveFace}) becomes: let $F$ be a fixed $k$-face of $\po$; then
\begin{equation} \label{eq:aveFace2}
 \frac{{\cal E}f_k(A\po)}{f_k(\po)} =  
\mbox{Prob}\{ AF \mbox{ is a $k$-face of } A\po \} .
 \end{equation}

Let $P$ be a polytope in $\RR^N$ and $x_0\in P$.  The vector 
$v$ is a feasible direction for $P$ at $x_0$ if $x_0+tv\in P$ for 
all sufficiently small $t>0$.  Let $\mbox{Feas}_{x_0}(P)$ denote the 
set of all feasible directions for $P$ at $x_0$.

\begin{lemma}\label{lem:face_transverse_po}
Let $x_0$ be a vector in $\RR_+^N$ with exactly $k$ nonzeros.  
Let $F$ denote the associated $k$-face of $\po$.  For an 
$n\times N$ matrix $A$, let $AF$ denote the image of $F$ under $A$.
The following are equivalent:

\begin{tabular}{ll}
(Survive($A,F,\po$)): & $AF$ is a $k$-face of $A\po$, \\
(Transverse($A,x_0,\po$)) & ${\cal N}(A)\cap 
\mbox{Feas}_{x_0}(\po)=\{0\}$.
\end{tabular}
\end{lemma}

We now develop the connections to the probabilities in
Wendel's theorem.

\begin{lemma}\label{lem:trans_wendel}
Let $x_0\in\po$ have $k$ nonzeros.  Let $A$ be $n\times N$ with $n<N$ 
have an orthant-symmetric null space with exchangeable columns.  Then
\[
\mbox{Prob}\{(\mbox{Transverse($A,x_0,\po$)})\;\;\mbox{Holds}\}= 1 - P_{N-n,N-k}
\]
\end{lemma}
\begin{proof}
Exchangeability of the columns implies that
\[
\mbox{Prob}\{(\mbox{Transverse($A,x_0,\po$)}) \;\;\mbox{Holds}\}
\]
does not depend on $x_0$ , but only on the number of nonzeros in $x_0$
and the size of $A$.
Therefore, let $k$ be the number of nonzeros in $x_0$, and set
\[
\pi_{k,n,N}\equiv
\mbox{Prob}\{(\mbox{Transverse($A,x_0,\po$)})  \;\;\mbox{Holds}\}.
\]

The matrix $A$ has its columns in general position.  Therefore we 
may construct a basis $b_i$ for its null space, $\na$, having exactly $N-n$ 
basis vectors.  The $N$ by $N-n$ matrix $B^T$ having the $b_i$ for its
columns generates every vector $w$ in $\na$ via a product of the form
$w=B^Tc$, where $c\in\RR^{N-n}$.

Without loss of 
generality, suppose the nonzeros of $x_0$ are in positions $i=N-k+1,\ldots,N$.
Then $\mbox{Feas}_{x_0}(\po)=\{v\; : \; v_1,\ldots v_{N-k}\ge 0\}$.
Condition $ \mbox{(Transverse($A,x_0,\po$))}$ can be restated as

\begin{equation}\label{eq:ineq}
\mbox{(Ineq)}\quad\left\{
\begin{array}{l} 
\mbox{The only vector }c\mbox{ satisfying} \\
(B^Tc)_i\ge 0,\quad i=1,\ldots,N-k, \\
\mbox{is the vector }c=0.
\end{array}
\right.
\end{equation}

Suppose the contrary to (Ineq), i.e. suppose there is a $c\ne 0$ 
solving (\ref{eq:ineq}).
Let now $\beta_i$ denote the $i$-th row of $B^T$, with $i=1,\ldots, N-k$.
Then (\ref{eq:ineq}) is the same as 
\[
\beta_i\cdot c\ge 0, \quad i=1,\ldots, N-k.
\]
Geometrically, this says that 

\hspace{2cm}
\begin{tabular}{l}
Each vector $\beta_i$, $i=1,\ldots, N-k$, \\
falls in the half-space $\beta \cdot c\ge 0$.
\end{tabular}

\noindent
Here $c$ is some fixed but arbitrary nonzero vector.
Thus the event \{(Ineq) does not hold\} is equivalent to 
the event

\hspace{2cm}
\begin{tabular}{l}
All the vectors $\beta_i$ with $i=1,\ldots, N-k$ \\
fall in \underline{some} half-space of $\RR^{N-n}$.
\end{tabular}

By our hypothesis, the vectors $\beta_i$ with $i=1,\ldots,N-k$ 
are drawn i.i.d. from a centrosymmetric distribution and are 
in general position.  We now invoke Wendel's Theorem, and it follows 
that 
\[
\pi_{k,n,N}= 1-P_{N-n,N-k}.
\]

\end{proof}

\subsection{Some Generalities about Binomial Probabilities}

The probability $P_{m,M}$ in Wendel's theorem
has a classical interpretation: it gives the probability of at most $m-1$
heads in $M-1$ tosses of a fair coin.  The usual Normal
approximation to the binomial tells us that
\[
   P_{m,M} \approx \Phi \left(  \frac{(m-1) - (M-1)/2}{\sqrt{(M-1)/4}}  \right), 
\]
with $\Phi$ the usual standard normal distribution
function $\Phi(x) = \int_{-\infty}^x e^{-y^2/2} dy/\sqrt{2\pi}$;
here the approximation symbol $\approx$ can be made
precise using standard limit theorems,
eg. appropriate for small or large deviations. In this expression,
the approximating normal has mean $(M-1)/2$ and standard deviation
$\sqrt{(M-1)/4}$.
There are three regimes of interest, for large $m$, $M$, and three
behaviors for $P_{m,M}$.
\bitem
 \item Lower Tail: $m   \ll  M/2 -\sqrt{M/4}$.  $P_{m,M} \approx 0$.
  \item Middle: $m \approx M/2$.  $P_{m,M} \in (0,1)$.
 \item  Upper Tail: $m  \gg  M/2 + \sqrt{M/4}$.  $P_{m,M} \approx 1$.
 \eitem

\subsection{Proof of Theorem \ref{thm:main_weak_orthant}}
\label{sec:asymp}

Using the correspondence $N-n \leftrightarrow m$, $N-k \leftrightarrow M$,
and the connection to Wendel's theorem, 
we have three regimes of interest:
\bitem
\item  $N-n \ll (N-k)/2$
\item $N-n  \approx (N-k)/2$
\item $N-n \gg (N-k)/2$
\eitem
In the proportional-dimensional framework, the above discussion
translates into three separate regimes, and separate
behaviors we expect to be true:
\bitem
\item Case 1: $\rho < \rho_W(\delta;\h)$.  $P_{N_n-n,N_n-k_n} \goto 0$.
\item Case 2:  $\rho = \rho_W(\delta;\h)$. $P_{N_n-n,N_n-k_n}  \in (0,1)$.
\item Case 3  $\rho > \rho_W(\delta;\h)$. $P_{N_n-n,N_n-k_n} \goto 1$.
\eitem

Case $2$ is trivially true, but it has no role in the statement of Theorem
 \ref{thm:main_weak_orthant}.  Cases 1 and 3 correspond exactly to the
 two parts of (\ref{eq:weak_main}) that we must prove.
 
To prove Cases 1 and 3, we need
an upper bound deriving from standard large-deviations analysis 
of the lower tail of the binomial.
\begin{lemma} \label{lem:regime1}
Let $N-n < (N-k)/2$.
\begin{equation} \label{eq:bnd_regime1}
  P_{N-n,N-k} \leq n^{3/2} \exp \left( N \psiwpo\left(\frac{n}{N},\frac{k}{n}\right) \right)
\end{equation}
where the exponent is defined as
\begin{equation}\label{eq:weak_exponent}
\psiwpo(\delta,\rho):=H(\delta)+\delta H(\rho)-H(\rho\delta)-(1-\rho\delta)\log 2
\end{equation}
with $H(\cdot)$ the Shannon Entropy (\ref{eq:Shannon})
\end{lemma}

\noindent
{\bf Proof.} Upperbounding the sum in $P_{N-n,N-k}$ by $N-n-1$ times 
${N-k-1 \choose N-n}$ we arrive at
\begin{equation}\label{eq:comb_weak}
P_{N-n,N-k}\le 2^{N-k-1}\frac{n-k}{N-k}\cdot (N-k+1)
{N\choose n} {n\choose k} {N\choose k}^{-1}.
\end{equation}
We can   bound ${m \choose \gamma\cdot m}$ for $\gamma<1$
using the Shannon entropy (\ref{eq:Shannon}):
\begin{equation}\label{eq:choosebound}
c_1 n^{-1/2} e^{mH(\gamma)}\le
{m\choose \gamma\cdot m} 
\le c_2 e^{mH(\gamma)}
\end{equation}
where  $c_1:=\frac{16}{25}\sqrt{2/\pi}$, 
$c_2:=5/4\sqrt{2\pi}$.
Recalling the definition of $\psiwpo$, we obtain  (\ref{eq:bnd_regime1}). \qed

We will now consider Cases 1 and 3, and prove the corresponding conclusion.


{\bf Case 1: $\rho<\rho_W(\delta;\po)$.}
The threshold function $\rho_W(\delta;\po)$ is defined as the zero level curve 
$\psi_W^H(\delta,\rho_W(\delta;\po))=0$; thus
for any $\rho$ strictly below $\rho_W(\delta;\po)$, the exponent 
$\psiwpo(\delta,\rho)$ is strictly negative.
Lemma \ref{lem:regime1} thus implies that 
$P_{N_n-n,N_n-k_n}\goto 0$ as $n\goto\infty$.

\begin{figure}[h!]
\begin{tabular}{c}
\psfig{figure=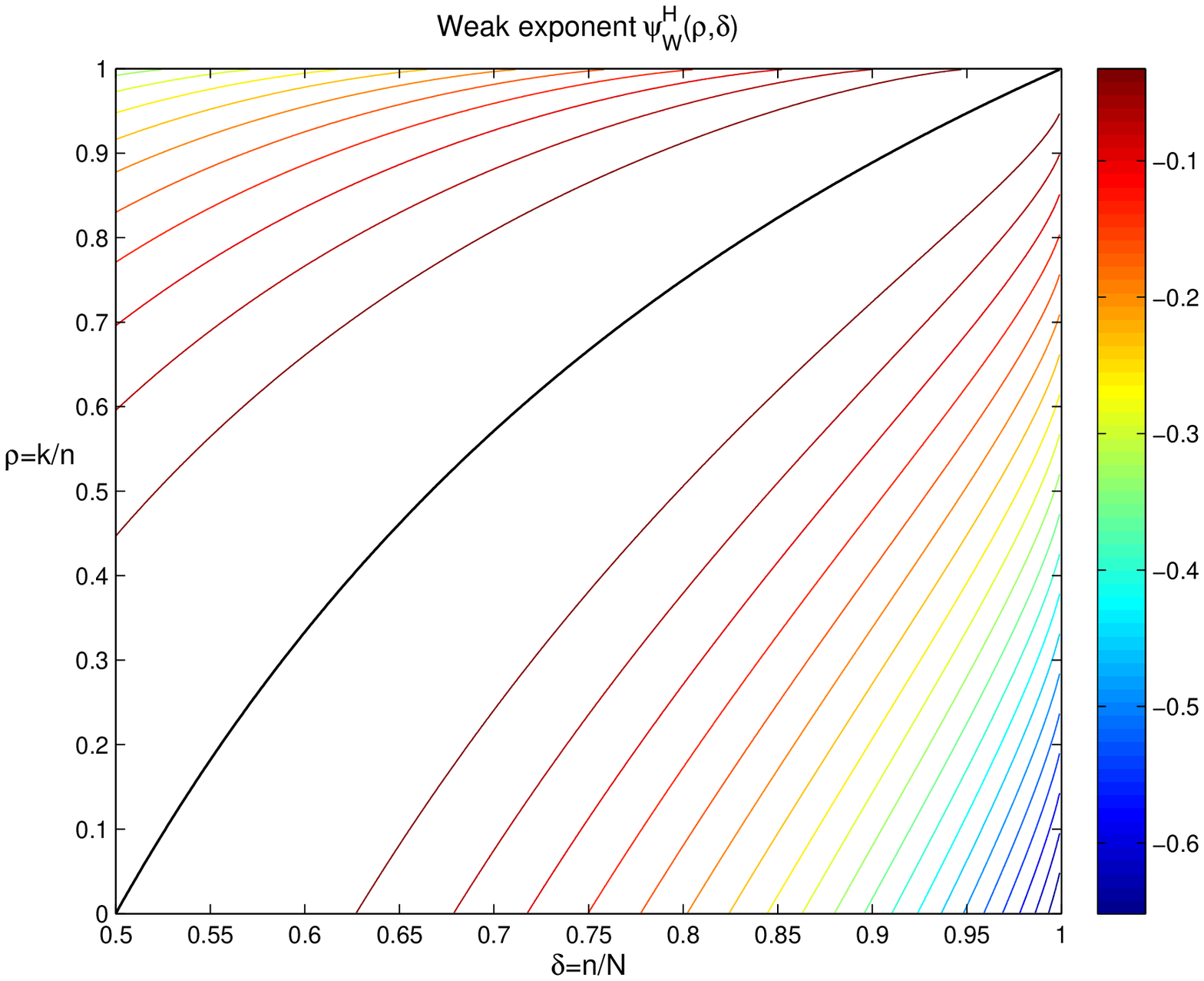,height=2.1in,width=4.3in} 
\end{tabular}
\caption{Exponent for the weak phase transition, $\psiwpo(\rho,\delta)$, 
(\ref{eq:weak_exponent}), 
which has its zero level curve at $\rho_W(\delta;\po)$, 
equation (\ref{eq:rhoW}). 
The projected hypercube has the same weak phase transition and exponent 
$\psiwh\equiv\psiwpo$.}\label{fig:weak_exponent}
\end{figure}

{\bf Case 3: $\rho>\rho_W(\delta;\po)$.}
Binomial probabilities have a standard symmetry (relabel every `head' outcome
as a 'tail', and vice versa). It follows that $P_{m,M} = 1 - P_{M-m,M}$.
We have $P_{N-k,N-n}=1-P_{N-k,n-k}$. In this case $N-n>(N-k)/2$, 
so Lemma \ref{lem:regime1} tells us that
$P_{N-k,n-k}\goto 0$ as $n\goto\infty$; we conclude $P_{N-k,N-n}\goto 1$ 
as $n\goto\infty$.

\subsection{Proofs of Theorems \ref{thm:main_weak_hypercube} and 
\ref{thm:second_weak_hypercube}}\label{sec:weak_hypercube_proof}

We derive the exact non-asymptotic result
Theorem \ref{thm:finite_samp_hyp} from  Theorem \ref{thm:isom}
by symmetry.
The limit results in Theorems 
\ref{thm:main_weak_hypercube} and 
\ref{thm:second_weak_hypercube} follow immediately from
asymptotic analysis of Section \ref{sec:asymp}.

We begin as before, relating face counts to probabilities of survival.
\begin{equation}\label{eq:ave_h}
 \frac{{\cal E}f_k(A\h)}{f_k(\h)} = 
\mbox{Ave}_F \left[ \mbox{Prob}\{ AF \mbox{ is a $k$-face of } A\h \}\right].
 \end{equation}
Here $\mbox{Ave}_F $ denotes the average over $k$-faces of $\h$.

As before, we assume exchangeable columns as a 
calculation device, allowing us to focus on one $k$-face, but 
compute the average.  Under exchangeability, for any fixed $k$-face $F$,
\begin{equation}\label{eq:ave_h2}
 \frac{{\cal E}f_k(A\h)}{f_k(\h)} = 
 \mbox{Prob}\{ AF \mbox{ is a $k$-face of } A\h \}.
 \end{equation}

We also again reformulate  matters in terms of transversal intersection.
\begin{lemma}\label{lem:face_transverse_h}
Let $x_0$ be a vector in $\h$ with exactly $k$ nonzeros.  
Let $F$ denote the associated $k$-face of $\h$.  For an 
$n\times N$ matrix $A$ the following are equivalent:
\begin{tabular}{ll}
(Survive($A,F,\h$)): & $AF$ is a $k$-face of $A\h$, \\
(Transverse($A,x_0,\h$)): & ${\cal N}(A)\cap \mbox{Feas}_{x_0}(\h) = \{0\}$.\\
\end{tabular}
\end{lemma}

We next connect the hypercube to the positive orthant.
Informally, the point is that the positive orthant 
in some sense shares faces
 with the "lower faces" of the hypercube.

Formally, let $x_0$ be a vector having $x(i) = 0, 1 \leq i \leq N-k-1$,
and $x(i)  = 1/2$, $N-k \leq i \leq N$.  Then $x_0$ belongs
to both $\h$ and $\po$. It makes sense to define 
the two cones $Feas_{x_0}(\h)$ and $Feas_{x_0}(\po)$
for this specific point $x_0$, and we immediately see
\[
  Feas_{x_0}(\h) = Feas_{x_0}(\po) .
\]
In fact this equality holds for all $x_0$ in the relative
interior of the $k$-face of $\h$ containing $x_0$. We conclude:

\begin{lemma} \label{lem:hyper_orth}
Let $F_{k,H}$ be the $k$-dimensional face of $\h$
consisting of all vectors $x$ with $x(i) = 0, 1 \leq i \leq N-k-1$,
and $0 \leq x(i) \leq 1$, $N-k \leq i \leq N$.  Let $F_{k,\RR_+}$ be the $k$-dimensional face of $\po$
consisting of all vectors $x$ with $x(i) = 0, 1 \leq i \leq N-k-1$,
and $0 \leq x(i)$, $N-k \leq i \leq N$.
Then
\begin{equation}
\mbox{Prob}\{ AF_{k,H} \mbox{ is a $k$-face of } A\h \} 
= \mbox{Prob}\{ AF_{k,\RR_+} \mbox{ is a $k$-face of } A\po \}.
\end{equation}
\end{lemma}

Combining (\ref{eq:ave_h}) and Lemma \ref{lem:hyper_orth}
we obtain the non-asymptotic Lemma \ref{thm:finite_samp_hyp}
from the corresponding non-asymptotic result for the positive orthant.


\subsection{Proof of Theorem \ref{thm:main_strong}}\label{sec:proof_strong}
$P_{N-n,N-k}$ is the probability that {\it one} fixed $k$-dimensional face  $F$
of $\po$ generates a $k$-face $AF$ of $A\po$.
The probability that {\em some} $k$-dimensional face generates a $k$-face
can be upperbounded, using Boole's inequality, by $f_k(\po)\cdot P_{N-n,N-k}$.

From (\ref{eq:choosebound}), (\ref{eq:bnd_regime1}), and $f_k(\po)={N \choose k}$ 
we have
\[
f_k(\po)\cdot P_{N-n,N-k}\le n^{3/2} \exp(N\psispo(\delta_n,\rho_n))
\]
where $\psispo$ was defined earlier in (\ref{eq:psis}), as
\begin{equation}\label{eq:psistrong}
\psispo(\delta,\rho):=H(\delta)+\delta H(\rho)-(1-\rho\delta)\log 2.
\end{equation}
Recall that for $\delta\ge 1/2$,  
$\rho_S(\delta;\po)$ is the zero crossing of $\psispo$.
For any $\rho<\rho_S(\delta;\po)$ we have 
$\psispo(\delta,\rho)<0$ and as a result (\ref{eq:strongLim}) follows.

\section{Contrasting the Hypercube with Other Polytopes}\label{sec:existing}

The theorems in Section \ref{sec:intro} contrast strongly with
existing results for other polytopes.

\subsection{Non-Existence of Weak Thresholds at $\delta < 1/2$}

Theorem \ref{thm:second_weak_hypercube} identifies a 
region of $(\frac{n}{N},\frac{k}{N})$ where the typical
random zonotope has nearly as many 
$k$-faces as its generating hypercube; in particular,  if $n < N/2$, it
has many fewer $k$-faces than the hypercube, for every $k$.    
This behavior at $n/N < 1/2$ is quite different from 
the behavior of typical random projections
of the simplex and the cross-polytope. Those polytopes
have  $f_k(AQ) \approx f_k(Q)$ for quite a large range of $k$ even at
relatively small values of $k/n$, \cite{DoTa07_CF}, see Figure 
\ref{fig:simplex_crosspolytope}.

\begin{figure}[h!]
\begin{tabular}{c}
\psfig{figure=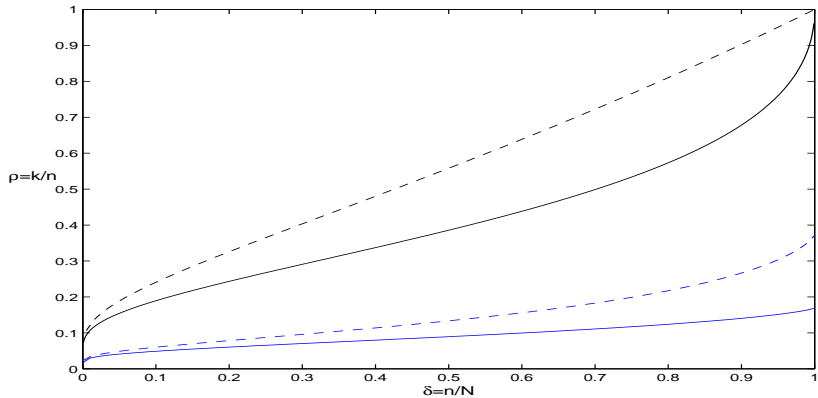,height=2.1in,width=4.3in} 
\end{tabular}
\caption{Weak thresholds for the simplex, $\rho_W(\delta;T^{N-1})$ 
(black-dash), and cross-polytope, $\rho_W(\delta;C^N)$ (black-solid).   
Consider sequences obeying the proportional-dimensional
asymptotic with parameters $\delta$, $\rho$. For $(\delta,\rho)$
below these curves, and for large $n$, each projected 
polytope has nearly as many $k$-faces as its generator; above these curves
the projected polytope has noticeably fewer. 
Strong thresholds for the simplex, $\rho_S(\delta;T^{N-1})$ (blue-dash), 
and cross-polytope, $\rho_S(\delta;C^N)$ (blue-solid).  For $(\delta,\rho)$
below these curves, and for large $n$, each projected 
polytope and its generator typically have exactly the same number 
of $k$-faces.}\label{fig:simplex_crosspolytope}
\end{figure}

\subsection{Non-Existence of Strong Thresholds for Hypercube}
Lemma \ref{thm:strong_zonotope} shows that projected zonotopes always 
have strictly fewer 
$k$-faces than their generators $f_k(A\h) < f_k(\h)$,
for every $n < N$.
this is again quite different from the situation with
the simplex and the cross-polytope,
where we can even have $n \ll N$ and still find $k$ for which
$f_k(AQ) = f_k(Q)$, \cite{DoTa07_CF}, see 
Figure \ref{fig:simplex_crosspolytope}.

\subsection{Universality of weak phase transitions}

For Theorems \ref{thm:main_weak_hypercube} and \ref{thm:second_weak_hypercube}, 
$A$ can be sampled from any
ensemble of random matrices having an 
orthant-symmetric and generic random null space.
Our result is thus {\em universal} across a 
wide class of matrix ensembles.   

In proving weak and strong threshold 
results for the simplex and cross-polytope,
we required $A$ to either be a random ortho-projector
or to have  Gaussian iid entries.  Thus, what we proved
for those families of regular polytopes  applies to a much more limited
range of matrix ensembles than
what has now been proven for hypercubes.    

Our empirical studies suggest that the
same ensembles of matrices which `work' for the
hypercube weak threshold also `work' for the
simplex and cross-polytope thresholds.
It seems to us that the universality across
matrix ensembles proven here may point to a much larger
phenomenon, valid also for other polytope families.  
For our empirical studies
see \cite{DoTa08_universality}.
 
In fact, even in the hypercube case, the weak threshold phenomenon
may be more general than what can be proven today;
it seems also to hold for some matrix ensembles
that may not have an orthant-symmetric null space.


\section{Contrasting  the Cone with the Hypercube}

The weak Cone threshold
depends very much more delicately on details about
$A$ than do the hypercube thresholds; it really makes
a difference to the results if the matrix $A$ is not `zero-mean'.

\subsection{The Low-Frequency Partial Fourier Matrix}
\label{sec:PartFourier}

Consider the special partial Fourier matrix made 
only of the $n$ lowest frequency entries.
\begin{corollary}\label{cor:fourier}
Assume $n$ is odd and let
\begin{equation}\label{eq:fourier}
\Omega_{ij}=\left\{ \begin{array}{ll} 
\cos\left(\frac{\pi(j-1)(i-1)}{N}\right) & i=1,3,5,\ldots,n \\
\sin\left(\frac{\pi(j-1)i}{N}\right) & i=2,4,6,\ldots,n-1.
\end{array}\right.
\end{equation}
Then 
\[
f_k(\Omega \po) = f_k(\po), \quad k=0,1, \dots ,  \frac{1}{2}(n-1) .
\]
\end{corollary}

This behavior is dramatically different than the case for random
$A$ of the type considered so far, and in some sense dramatically better.

Corollary \ref{cor:fourier} is closely connected with the classical question of
{\it neighborliness}.    There are famous polytopes which can be generated
by projections $AT^{N-1}$ and have exactly as many
$k$-faces  as $T^{N-1}$  for $k \leq \lfloor n/2 \rfloor$.  A standard example
is provided by the matrix $\Omega$ defined in (\ref{eq:fourier}); it
obeys $f_k(\Omega T^{N-1}) =  f_k(T^{N-1})$, $0 \leq k \leq \lfloor n/2 \rfloor$.
(There is a vast literature touching in some way on the phenomenon
$f_k(\Omega T^{N-1}) =  f_k(T^{N-1})$. In that literature,
the polytope $\Omega T^{N-1}$ is usually called a {\it cyclic polytope},
and the columns of $\Omega$ are called points of the
{\it trigonometric moment curve}; see standard references
\cite{Gru67,Ziegler}).

Hence the matrix $\Omega$ 
offers both $f_k(\Omega T^{N-1}) = f_k(T^{N-1})$ and $f_k(\Omega\po) =  f_k(\po)$
for $0 \leq k \leq \lfloor n/2 \rfloor$.  This is exceptional.
For random $A$ of the type discussed in earlier sections,
there is a large disparity between the sets of triples $(k,n,N)$
where  $f_k(AT^{N-1}) = f_k(T^{N-1})$ -- this happens for
$k/n < \rho_S(n/N;T^{N-1})$  -- and
those where  $f_k(A\po) =  f_k(\po)$ -- this happens for
$k/n < \rho_S(n/N;\po)$.   These two strong thresholds
are displayed in Figures \ref{fig:simplex_crosspolytope} and \ref{fig:phase} respectively.

Even if we relax our notion of agreement of face counts to
weak agreement, the collections of triples where
 $f_k(AT^{N-1}) \approx f_k(T^{N-1})$ and $f_k(A\po) \approx f_k(\po)$
 are very different, because the two curves
$ \rho_W(n/N;T^{N-1})$  and $ \rho_W(n/N;\po)$ 
are so dramatically different, particularly at $n < N/2$.
 
\subsection{Adjoining a Row of Ones to $A$}

An important feature of the random matrices $A$ studied
earlier is that their random nullspace is orthant symmetric.
In particular, the positive orthant plays no distinguished role
with respect these matrices.  On the other hand,
the partial Fourier matrix $\Omega$ constructed 
in the last subsection contains a row of ones,
and thus the positive orthant has a distinguished role
to play for this matrix.  Moreover, this distinction
is crucial; we find empirically that {\it removing}
the row of ones from $\Omega$ causes the conclusion 
of Corollary \ref{cor:fourier}
to fail drastically.

Conversely,
consider the matrix $\tilde{A}$ obtained by {\it adjoining} a row of $N$ ones
to some matrix A:
\[
       \tilde{A} = \left[ \begin{array}{l} 1 \\ A \end{array} \right].
\]
Adding this row of ones to a random matrix causes a drastic shift in the strong
and weak thresholds. The following is proved in Section \ref{sec:MoreProofs}.

\begin{theorem} \label{thm:adj-shift-thresh}
Consider the proportional-dimensional asymptotic
with parameters $\delta, \rho$ in $(0,1)$.
Let the random $n-1$ by $N$ matrix $A$ have iid standard normal entries. 
Let $\tilde{A}$ denote the corresponding $n$ by $N$ matrix whose
first row is all ones and whose remaining rows are identical to those
of $A$. Then 
\begin{equation} \label{eq:altpos_weak}
\lim_{n \goto \infty}  \frac{{\cal E}f_k(\tilde{A}\po)}{f_k(\po)} = 
\left\{\begin{array}{cl}
1, & \rho<\rho_W(\delta,T^{N-1}) \\
<1, & \rho>\rho_W(\delta,T^{N-1})
\end{array}\right. .
\end{equation}
\begin{equation} \label{eq:altpos_strong}
\lim_{n \goto \infty}  P \{ f_k(\tilde{A}\po) = f_k(\po)\} = 
\left\{\begin{array}{cl}
1, & \rho<\rho_S(\delta,T^{N-1}) \\
0, & \rho>\rho_S(\delta,T^{N-1})
\end{array}\right. .
\end{equation}
\end{theorem}

Note particularly the {\it mixed} form of this relationship.
Although the conclusions concern the behavior of faces of
the randomly-projected {\it orthant}, the thresholds are those
that were previously obtained for the randomly-projected {\it simplex}.

Since there is such a dramatic difference between
$\rho(\delta,T^{N-1})$ and $\rho(\delta,\po)$, the single
row of ones can fairly be said to have a huge effect.
In particular, the region 'below' the simplex weak phase transition 
$\rho_W(\delta,T^{N-1})$ comprises $\approx 0.5634$ of the 
$(\delta,\rho)$ parameter area, and the hypercube weak phase 
transition $\rho_W(\delta,\h)$ comprises $1-\log 2\approx 0.3069$.

\section{Application: Compressed Sensing}\label{sec:cs}

Our face counting results can all be reinterpreted
as statements about  ``simple'' solutions of underdetermined
systems of linear equations. This reinterpretation allows
us to make connections with numerous problems
of current interest in signal processing, information theory,
and probability. The reinterpretation follows from the two 
following lemmas, which are restatements of 
Lemmas \ref{lem:face_transverse_po} and \ref{lem:face_transverse_h}, 
rephrasing the notion of (Transverse($A,x_0,Q$))
with the all but linguistically equivalent (Unique($A,x_0,Q$)).
For proofs of Lemmas \ref{lem:pos_unique} and \ref{lem:hyp_unique} 
see the proofs of 
Lemmas \ref{lem:face_transverse_po} and \ref{lem:face_transverse_h}.

\begin{lemma}\label{lem:pos_unique}
Let $x_0$ be a vector in $\RR_+^N$ with exactly $k$ nonzeros.  
Let $F$ denote the associated $k$-face of $\po$.  For an 
$n\times N$ matrix $A$, let $AF$ denote the image of $F$ under $A$
and $b_0 = Ax_0$ the image of $x_0$ under $A$.
The following are equivalent:

\begin{tabular}{ll}
$\mbox{(Survive($A,F,\po$))}$: & $AF$ is a $k$-face of $A\po$, \\
$\mbox{(Unique($A,x_0,\po$))}$: & The system $b_0=Ax$ has a unique solution in $\po$. 
\end{tabular}
\end{lemma}

\begin{lemma}\label{lem:hyp_unique}
Let $x_0$ be a vector in $\h$ with exactly $k$ entries 
strictly between the bounds $\{0, 1\}$.
Let $F$ denote the associated $k$-face of $\h$.  For an 
$n\times N$ matrix $A$, let $AF$ denote the image of $F$ under $A$
and $b_0 = Ax_0$ the image of $x_0$ under $A$.
The following are equivalent:

\begin{tabular}{ll}
$\mbox{(Survive($A,F,\h$))}$: & $AF$ is a $k$-face of $A\h$, \\
$\mbox{(Unique($A,x_0,\h$))}$: & The system $b_0=Ax$ has a unique solution in $ \h$. 
\end{tabular}
\end{lemma}

Note that the systems of linear equations referred
to in these lemmas are underdetermined: $n < N$.
Hence these lemmas identify 
conditions on underdetermined system of
linear equations, such that,  when  the
solution is known to obey certain constraints,  there are many
cases where this seemingly weak a priori knowledge 
in fact uniquely determines
the solution.  The first result can be paraphrased as
saying that nonnegativity constraints can be very powerful,
if the object is known to have relatively few nonzeros;
the second result says that upper and lower bounds
can be very powerful, provided those bounds are active
in most cases.

These results provide a theoretical vantage point on
an area of recent intense interest in signal processing,
appearing variously 
under the labels ``Compressed Sensing'' or ``Compressive Sampling''.

In many practical applications of scientific and
engineering signal processing -- spectroscopy is one example --
one can obtain $n$ linear measurements
of an object $x$, obtaining data
$b = Ax$; here the rows of the matrix $A$ give the
linear response functions of the measurement devices.  
We wish to reconstruct  $x$, knowing only the measurements $b$, 
the measurement matrix $A$, 
and various {\t a priori} constraints on $x$.

It could be very useful to be able
to do this in the case $n < N$, allowing us to save measurement time
or other resources.   
This seems hopeless, because the linear system is underdetermined;
but the above lemmas show that there is some fundamental
soundness to the idea that we can have $n < N$ and still
reconstruct.  We now spell out the consequences of these lemmas
in more detail.

\subsection{Reconstruction Exploiting Nonnegativity Constraints}
Many practical applications, such as spectroscopy
and astronomy, the object $x$ to be recovered
is known {\it a priori} to be nonnegative.
We wish to reconstruct  the unknown $x$, knowing only the 
linear measurements $b=Ax$, the matrix $A$, and
the constraint $x \in \po$.  

Let $J(x)$ be some function of $x$. Consider the positivity-constrained
variational problem
\[
   (Pos_{J}) \qquad \min J(x) \qquad \mbox{ subject to } b = Ax, \qquad x \in \po.    
\]
Let $pos_J(b,A)$ denote any solution of the problem instance $(Pos_J)$
defined by data $b$ and matrix $A$.

Typical variational functions $J$ include
\bitem
 \item Sparsity: $\|x\|_{\ell^0}:=  \#\{i: x > 0 \}$.
 \item Size:  $1'x$.
 \item negEntropy: $\sum x(j) \log(x(j))$
 \item Energy: $ \sum x(j)^2$
 \eitem
This framework contains as special cases the popular signal processing
 methods of maximum entropy reconstruction
 and nonnegative least-squares reconstruction.
  
We conclude the following:
\begin{corollary}
Suppose that
\[
    f_k(A \po) = f_k(\po) .
\]
Let $x_0 \geq 0$ and $\|x_0\|_{\ell^0} \leq k$. For the
problem instance defined by $b = Ax_0$
\[
     pos_J(b,A) = x_0 .
\]
\end{corollary}
In words: under the given conditions on the face numbers,
{\it any} variational prescription which imposes nonnegativity constraints
will correctly recover the $k$-sparse solution in
{\it any} problem instance where such a $k$-sparse solution exists.
This may seem surprising;  as $n < N$, the system of linear
equations is underdetermined yet we correctly find a sparse
solution if it exists.

Corresponding to this `strong' statement is a `weak' statement.
Consider the following probability measure on $k$-sparse problem
instances.
\bitem
 \item Choose a random subset $I$ of size $k$ from $\{1,\dots,N\}$,
  by $k$ simple random draws without replacement.
  \item Set the entries of $x_0$ not in the selected subset to zero.
  \item Choose the entries of $x_0$ in the selected set $I$
   from some fixed joint distribution $\psi_{I}$ supported in $(0,1)^k$.
   \item Generate the problem instance $b = Ax_0$.
 \eitem
We speak of drawing a $k$-sparse random problem instance at random.
\begin{corollary}
Suppose that for some $\eps \in (0,1)$.
\[
    f_k(A \po) \geq (1- \eps) \cdot  f_k(\po) .
\]
For $(b,A)$ a problem instance drawn at random,
as above:
\[
 \mbox{Prob} \{ pos_J(b,A) = x_0 \}  \geq (1-\eps).
 \]
\end{corollary}
In words: under the given conditions on the face
lattice, {\it any} variational prescription which imposes nonnegativity constraints
will correctly succeed to recover the $k$-sparse solution
in at least a fraction $(1-\eps)$ of all
$k$-sparse problem instances.
This may seem surprising; since $n < N$, the system of linear
equations is underdetermined, and yet, we typically find a sparse
solution if it exists.

Here are some simple applications:

\bitem
\item In the proportional-dimensional framework, consider triples $(k_n,n,N_n)$
with parameters $\delta, \rho$. Let $A$ denote an $n$ by $N_n$  matrix 
having random nullspace which is orthant symmetric and generic.
\bitem
\item  If the parameters $\delta, \rho$ name a point  'below' 
the {\it orthant weak threshold} $\rho_W(\delta;\po)$, 
then for the vast majority of $k_n$-sparse vectors,
any variational method will correctly recover the vector.
\item If the parameters $\delta, \rho$ name a point  'below' the 
{\it  orthant strong threshold} $\rho_S(\delta;\po)$, 
then for large enough $n$, every $k_n$-sparse vector
can be correctly recovered by any variational method imposing positivity
constraints.
\eitem
\item In the proportional-dimensional asymptotic, consider triples $(k_n,n,N_n)$
with parameters $\delta, \rho$. Let $A_0$ denote an $n-1$ by $N_n$  matrix 
having iid standard normal entries. And let $A$ denote the $n$ by $N_n$ matrix
formed by adjoining a row of ones to $A_0$.
\bitem 
\item
If the parameters $\delta, \rho$ name a point  'below' the 
{\it simplex weak threshold} $\rho_W(\delta;T^{N-1})$, 
then for the vast majority of $k_n$-sparse vectors,
any variational method will correctly recover the sparse vector.
\item If the parameters $\delta, \rho$ name a point  'below' 
the {\it simplex strong threshold} $\rho_S(\delta,T^{N-1})$, 
then for large enough $n$, every $k_n$-sparse vector
can be correctly recovered by any variational method imposing positivity
constraints.
\eitem
\item Let $A$ denote the $n$ by $N$ partial Fourier matrix
built from low frequencies and 
called $\Omega$ in Section \ref{sec:PartFourier}.
Every $\lfloor n/2 \rfloor$-sparse vector
will be correctly recovered by any variational method imposing positivity
constraints.
\eitem

Hence in positivity-constrained reconstruction problems
where the object to be recovered is zero in most entries --
an assumption which approximates the truth in many problems
of spectroscopy and astronomical imaging  \cite{DJHS92},
we can work with fewer than $N$ samples. 
The above paragraphs show that it matters a great
deal what matrix $A$ we use. Our preference order:
\begin{quotation}
$\Omega$ is better than  the random matrix, 
$\tilde{A}$ is better than a random zero-mean matrix $A$.
\end{quotation}

(These results extend and generalize results which were previously obtained
by the authors in \cite{DoTa05_signal}, in the case where $J(x) = 1'x$, and
by the first author and coauthors in \cite{DJHS92}; 
see also Fuchs \cite{Fuchs} and Bruckstein, Elad, and Ziubulevsky 
\cite{Elad_nonnegative}.)

\subsection{Reconstruction Exploiting Box Constraints}

Consider again the problem of reconstruction from
measurements $b = Ax$, but this time assuming the object $x$
obeys box-constraints: $0 \leq x(j) \leq 1$, $1 \leq j \leq N$.
Such constraints can arise for example in infrared absorption
spectroscopy and in binary digital communications.

We define the box-constrained variational problem
\[
   (Box_{J}) \qquad \min J(x) \qquad \mbox{ subject to } b = Ax, \qquad  0 \leq x(j) \leq 1, \quad j=1,\dots, N.    
\]
Let $box_J(b,A)$ denote any solution of the problem instance $(Box_J)$
defined by data $b$ and matrix $A$.

In this setting, the notion corresponding to 'sparse' is 'simple'.
We say that a vector $x$ is {\it $k$-simple}
 if at most $k$ of its entries differ from the bounds
 $\{0,1\}$.
Here, the interesting functions $J$ 
penalize deviations from simple structure; they include:
\bitem
 \item Simplicity: $  \#\{i: x(i) \not\in \{0,1\}  \}$.
  \item Violation Energy: $\sum x(j)  (1-x(j))$
 \eitem

\begin{corollary}
Suppose that
\[
    f_k(A \h) = f_k(\h) .
\]
Let $x_0$ be a  $k$-simple vector obeying the box
constraints $0 \leq x_0  \leq 1$.
For the problem instance defined by $b = Ax_0$,
\[
    box_J(b,A) = x_0 .
\]
\end{corollary}
In words: under the given conditions on the face lattice,
{\it any} variational prescription which imposes box constraints,
when presented with a problem instance where there is a $k$-simple solution,
will correctly recover the $k$-simple solution.

Corresponding to this `strong' statement is a `weak' statement.
Consider the following probability measure on problem instances having
 $k$-simple solutions.  
Recall that $k$-simple vectors  have all entries equal to $0$ or $1$
except at $k$ exceptional locations.
\bitem
\item Choose the subset $I$ of $k$ exceptional entries uniformly
at random from the set $\{1,\dots, N\}$ without replacement;
\item Choose the nonexceptional
entries to be either $0$ or $1$ based on tossing a fair coin.
\item Choose 
 the values of the exceptional $k$ entries 
 according to a joint probability measure $\psi_I$ supported in $(0,1)^k$.
 \item Define the problem instance $b = Ax_0$.
\eitem

\begin{corollary}
Suppose that for some $\eps \in (0,1)$.
\[
    f_k(A \h) \geq (1- \eps) \cdot  f_k(\h) .
\]
Randomly sample a problem instance $(b,A)$
using the method just described.
\[
P \{  box_J(b,A) = x_0 \} \ge (1-\eps) .
\]
\end{corollary}
In words: under the given conditions on the face
lattice, {\it any} variational prescription which imposes box constraints
will correctly recover at least a fraction $(1-\eps)$ of all
underdetermined systems generated by the matrix $A$
which have  $k$-simple solutions.

Here is a simple application.
In the proportional-dimensional asymptotic framework, consider triples $(k_n,n,N_n)$
with parameters $\delta, \rho$. Let $A$ denote an $n$ by $N_n$  matrix 
having random nullspace which is orthant symmetric and generic.
If the parameters $\delta, \rho$ name a point  'below' the hypercube
weak threshold, then for the vast majority of $k_n$-simple vectors,
any variational method imposing box constraints
will correctly recover the vector.

In the hypercube case, to our knowledge, there is no phenomenon comparable
to that which arose in the positive orthant with the 
special constructions $\Omega$ and $\tilde{A}$.

Consequently, the hypercube weak threshold is the best known
general result on the ability to undersample by exploiting box constraints.
In particular, the difference between the weak simplex threshold and
the weak hypercube threshold has this interpretation:
\begin{quotation}
A given degree $k$ of sparsity of a nonnegative object is
much more powerful than that same degree simplicity of a
box-constrained object.
\end{quotation}
Specifically, {\it we shouldn't expect to be able to undersample
a typical box-constrained object by more than a factor of $2$}
and then reconstruct it using some garden-variety variational
prescription.  In comparison, the last section showed that we can severely
undersample very sparse nonnegative objects.

Because box constraints are of interest in important areas
of signal processing, it seems that much more attention
should be paid to thresholds associated with the hypercube.


\section{Additional Proofs}\label{sec:MoreProofs}

\subsection{Proof of Lemma \ref{lem:face_transverse_po}}

Let $b_0:=Ax_0$.

Assume (Survive($A,F,\RR^N_+$)), that $AF$ is a $k$-face of 
$A\RR^N_+$.  General position of $A$ implies that 
$AF$ is a simplicial cone of dimension $k-1$, and that there 
exists a unique $x\in\RR^N_+$ satisfying $Ax=b_0$, with $x_0$ being that 
solution.  We now assume $\exists\;\nu\in\na\cap Feas_{x_0}(\RR^N_+)\ne 0$.
Then $\exists\;\epsilon>0$ small enough such that
$z_0:=x_0+\epsilon\nu\in\RR^N_+$.  This $z_0$ satisfies $Az_0=b_0$, 
in contradiction to the uniqueness condition previously stated, 
therefor $\na\cap Feas_{x_0}(\RR^N_+)=\{0\}$.

For the converse direction, assume (Transverse($A,x_0,\RR^N_+$)), 
that $\na\cap Feas_{x_0}(\RR^N_+)=\{0\}$.  Assume $AF$ is not a 
$k$-face of $A\RR^N_+$, that is $AF$ is interior to $A\RR^N_+$.  
As $A$ projects the interior of $\RR^N_+$ to the complete interior 
of $A\RR^N_+$, $\exists\; z_0\in\RR^N_+$ with $z_0>0$ with
$Az_0=b_0$.  The difference $\nu:=z_0-x_0\ne 0$,  but 
$\nu\in\na\cap Feas_{x_0}(\RR^N_+)$ contradicting the Transverse 
assumption, implying $AF$ is a $k$-face of $A\RR^N_+$.

\qed

\subsection{Proof of Lemma \ref{lem:face_transverse_h}}

This proof follows similarly to that of Lemma
\ref{lem:face_transverse_po} and is omitted.

\subsection{Proof of Lemma \ref{lem:hyper_orth}}\label{sec:orth_proof}

For points $x_0$ on $k$-faces of $\h$ that are also $k$-faces of 
$\po$ they share the same feasible set
\[
Feas_{x_0}(\h)=Feas_{x_0}(\po)
\]
and by Lemmas \ref{lem:face_transverse_po} and 
\ref{lem:face_transverse_h} the probabilities 
of (Survive($A,F,Q$)) for $Q=\po,\h$ must be equal. 
Consider a point $x_0$ on a $k$-face of $\h$ that is not  a
$k$-face of $\po$; without loss of generality, due to 
column exchangeability, let
\[
x_0(i)=\left\{ \begin{array}{cl}
0 & i=1,\ldots,\ell \\
1 & i=\ell+1,\ldots,N-k \\
1/2 & i=N-k+1,\ldots N
\end{array}\right.
\]
Then Feas$_{x_0}(\h)=\{\nu: \nu_1,\ldots,\nu_{\ell}\ge 0, 
\nu_{\ell+1},\ldots,\nu_{N-k}\le 0\}$.  Following the 
proof of Lemma \ref{lem:trans_wendel}, condition
(Transverse($A,x_0,\h$)) can be restated as 
\begin{equation}\label{eq:ineq_h}
(\mbox{Ineq } \h)\quad\left\{
\begin{array}{l} 
\mbox{The only vector }c\mbox{ satisfying} \\
(B^Tc)_i\ge 0,\quad i=1,\ldots,\ell, \\
(B^Tc)_i\le 0,\quad i=\ell+1,\ldots,N-k, \\
\mbox{is the vector }c=0
\end{array}
\right.
\end{equation}
where $B$ is the orthogonal complement of $A$.

Orthant symmetry of $B$ states that the sign of 
$(B^Tc)_i$ is equiprobable; consequently, the probability of 
the event named (\ref{eq:ineq_h}) is independent of $\ell$, and is 
in fact equal to the probability of the event named in (\ref{eq:ineq}).

\qed

\subsection{Proof of Corollary \ref{cor:fourier}}

The result is a corollary of \cite[Theorem 3, pp. 56]{DJHS92}.
However, it may require effort on the part of readers
to see this, so we select the key step from the proof 
of Theorem 3, \cite[Lemma 2, pp. 63]{DJHS92}, 
and use it directly within the framework of this paper.
\newcommand{\cH}{{\cal H}}
\newcommand{\cL}{{\cal L}}

As $n$ is odd, write  $n = 2m+1$ where $m$ is an integer.
The range of the matrix $\Omega$ is the span of all
Fourier frequencies from 0 to $\pi (m-1) /N$.
In accord with terminology in electrical engineering,
this space of vectors with be called the space of Lowpass
sequences $\cL(m)$.
The nullspace of $\Omega$ is the span of all Fourier
frequencies from $\pi m/N$ to $\pi/N$. It will be called
the space of Highpass sequences $\cH(m)$.

We have the following:

\begin{lemma} \cite{DJHS92}
Every sequence in $\cH(m)$ has at least $m$ negative entries.
\end{lemma}

Recall condition $(\mbox{Transverse}(\Omega,x_0,\po))$.
If $x_0$ has $k$ nonzeros,
then vectors in $Feas_{x_0}(\po)$  have at most
$k$ negative entries.  But vectors in $\cN(\Omega) = \cH(m)$ have
at least $m$ negative entries.  Therefore,
if $m > k$, $(\mbox{Transverse}(\Omega,x_0,\po))$ must hold.

By Lemma 2.2, every $(\mbox{Survive}(\Omega,F,\po))$ must hold
for every $k$-face with $k < m$. Hence $f_{k-1} (\Omega\po)  = f_{k-1}(\po)$,
for $k \leq m = \frac{1}{2}(n-1)$. \qed

\subsection{Proof of Theorem \ref{thm:adj-shift-thresh}}

The Theorem is an immediate consequence of the
following identity.

\begin{lemma}
Suppose that the row vector $1$ is
not in the row span of $A$. Then
\[
       f_k( \tilde{A} \po ) = f_{k-1}( A T^{N-1}),  0 < k < n.
\]
\end{lemma}

\begin{proof}
We observe that there is a  natural bijection between
$k$-faces of $\po$ and the $k-1$-faces of $T^{N-1}$.
The $k-1$-faces of $T^{N-1}$ are in bijection with the corresponding support sets of cardinality $k$:
i.e. we can identify with each $k$-face $F$ the union $I$ of all supports of all members of the face.
Similarly to each support set $I$ of cardinality $k$ there is a unique $k$-face $\tilde{F}$ of $\po$
consisting of all points in $\po$
whose support lies in $I$.   Composing bijections $F \leftrightarrow I \leftrightarrow \tilde{F}$
we have the bijection $F \leftrightarrow \tilde{F}$.

Concretely, let $x_0$ be a point in the relative interior of
some $k-1$-face $F$ of $T^{N-1}$.  Then $x_0$ has $k$ nonzeros.
$x_0$ is also in the relative interior of the $k$-face $\tilde{F}$ of $\po$
Conversely, let $y_0$ be a point in the relative interior
of some $k$-face of $\po$; then $x_0 = (1'y_0)^{-1} y_0$
is a point in the relative interior of a $k-1$-face of $T^{N-1}$.

The last two paragraphs show that for each pair
of corresponding faces $(F,\tilde{F})$,
we may find a point  $x_0$ in both the relative interior
of  $\tilde{F} $ and also of the relative interior
of $F$. For such $x_0$,
\[
 \mbox{Feas}_{x_0}(\po) = \mbox{Feas}_{x_0} (T^{N-1}) + lin(x_0).
\]
Clearly $\cN(\tilde{A}) \cap lin(x_0) = \{ 0 \}$, because $1'x_0  > 0$.
We conclude that the  following are equivalent:

\begin{tabular}{ll}
(Transverse($A,x_0,T^{N-1}$)) & ${\cal N}(A)\cap \mbox{Feas}_{x_0}(T^{N-1})=\{0\}$.\\
(Transverse($\tilde{A},x_0,\po$)) & ${\cal N}(\tilde{A})\cap \mbox{Feas}_{x_0}(\po)=\{0\}$. 
\end{tabular}

Rephrasing \cite{DoTa05_signal}, the following are equivalent for $x_0$ a
point in the relative interior of $F$:

\begin{tabular}{ll}
(Survive($A,F,T^{N-1}$)) & $AF$ is a $k-1$-face of $A T^{N-1}$, \\
(Transverse($A,x_0,T^{N-1}$)) & ${\cal N}(A)\cap \mbox{Feas}_{x_0}(T^{N-1})=\{0\}$.
\end{tabular}

We conclude that for two corresponding faces $F$, $\tilde{F}$,
the following are equivalent:

\begin{tabular}{ll}
(Survive($A,F,T^{N-1}$)): & $AF$ is a $k-1$-face of $A T^{N-1}$, \\
(Survive($\tilde{A},\tilde{F},\po$)): & $\tilde{A}\tilde{F}$ is a $k$-face of $\tilde{A} \po$ .
\end{tabular}

Combining this with the natural bijection $F \leftrightarrow \tilde{F}$, the lemma is proved.
\end{proof}

Acknowledgments.

Art Owen suggested that we pay attention to Wendel's Theorem.
We also thank Goodman, Pollack, and Schneider for providing 
scholarly background.

\bibliography{neigh}

\bibliographystyle{amsplain}
\end{document}